\documentclass{article}

\usepackage{comment}
\usepackage{hyphenat}
\usepackage{amsmath}
\usepackage{amsthm}
\usepackage{amssymb}
\usepackage{graphics}
\usepackage{pstricks}
\newtheorem{theorem}{Theorem}
\newtheorem{definition}{Definition}
\newtheorem{lemma}{Lemma}
\newtheorem{proposition}{Proposition}
\newtheorem{corollary}{Corollary}
\newtheorem{example}{Example}
\newtheorem{remark}{Remark}

\title{Irreducibility of configurations}           

\author{Klara Stokes}     
\date{University of Sk\"ovde, School of Engineering Sciences, \\54128 Sk\"ovde, Sweden\footnote{This article was partially written during postdoc positions at Universitat Oberta de Catalunya, Barcelona, Spain and Link\"oping University, Link\"oping, Sweden}          
}    

\begin{document}
\maketitle                     


\begin{abstract}
In a paper from 1886, Martinetti enumerated small $v_3$\hyp{}configurations. One of his tools was a construction that permits to produce a $(v+1)_3$\hyp{}configuration from a  $v_3$\hyp{}configuration. 
He called configurations that were not constructible in this way irreducible configurations. According to his definition, the irreducible configurations are Pappus' configuration and four infinite families of configurations.   
In 2005, Boben defined a simpler and more general definition of irreducibility, for which only two $v_3$\hyp{}configurations, the Fano plane and Pappus' configuration, remained irreducible. 
The present article gives a generalization of Boben's reduction for both balanced and unbalanced $(v_r,b_k)$\hyp{}configurations, and proves several general results on augmentability and reducibility. 
Motivation for this work is found, for example, in the counting and enumeration of configurations. 
\end{abstract}


\section{Introduction}
An incidence geometry is a triple $(P,L,I)$  where $P$ is a set of 'points', $L$ is a set of 'blocks', and $I$ is an incidence relation between the elements in $P$ and $L$. The line spanned by two points $p_1$ and $p_2$ is the intersection of all blocks containing both $p_1$ and $p_2$. 
When there are at most one block containing $p_i$ and $p_j$ for all pairs of points, then we may identify the blocks with the lines. Incidence geometries with this property are called partial linear spaces. 

If a point $p$ and a line $l$  are incident, then we say that $l$ goes through $p$, or that $p$ is on $l$. 
We say that a pair of lines that goes through the same point $p$ meet or intersect in $p$.  

A combinatorial configuration is a partial linear space in which there are $r$ lines through every point and $k$ points on every line~\cite{Davydov,Gropp,Grunbaum}.  
We will use the notation $(v_r,b_k)$\hyp{}configuration to refer to a combinatorial configuration 
with $v$ points, $b$ lines, $r$ lines through every point and  $k$ points on every line. 
The four parameters $(v_r,b_k)$ are redundant so that there is only need for the three parameters $(d,r,k)$,  where
$d:=\frac{v\gcd(r,k)}{k}=\frac{b\gcd(r,k)}{r}=\frac{vr}{\operatorname{lcm}(r,k)}=\frac{bk}{\operatorname{lcm}(r,k)}$ is an integer associated to the configuration that determines the number of points and lines. We will refer to $(d,r,k)$ as the reduced parameter set of the $(v_r,b_k)$\hyp{}configuration. 
When $v$ and $b$ are not known or not important, we will also use the notation $(r,k)$\hyp{}configuration.

We say that a  configuration is balanced if $r=k$. This implies that the number of points equals the number of lines and the associated integer, so $d=v=b$. In this case, we will use the notation $v_k$\hyp{}configuration. In the literature, configurations with this property are also called symmetric.  
When the configuration is unbalanced, i.e. when  $r\neq k$, then $v$, $b$ and $d$ are all different. Examples of balanced and unbalanced configurations are given in Figure~\ref{fig:conf}.

\begin{figure}
\begin{center}
\begin{tabular}{cccc}
\resizebox{.2\textwidth}{!}{

\begin{pspicture}(0,0)(4, 3.48)
  \psline[showpoints=true](0,0)(1,1.73204)(2,3.4641)
  \psline[showpoints=true](0,0)(2,0)(4,0)
  \psline[showpoints=true](4,0)(3,1.73204)(2,3.4641)
  \pscircle(2,1.1547){1.1547}
  \psline[showpoints=true](0,0)(2,1.1547)(3,1.73204)
  \psline[showpoints=true](4,0)(2,1.1547)(1,1.73204)
  \psline[showpoints=true](2,3.4641)(2,1.1547)(2,0)
  \end{pspicture}
}
&
\resizebox{.2\textwidth}{!}{
\begin{pspicture}(0.5, 0.5)(4.5, 4.5)
\psline[showpoints=true](0.5,1.5)(4.5,0.5)
\psline[showpoints=true](0.5,3.5)(4.5,4.5)
\psline[showpoints=true](1.3, 2.5)(1.85,2.5)(3.35,2.5)
\psline[showpoints=true](0.5, 1.5)(4.5, 4.5)
\psline[showpoints=true](0.5, 3.5)(4.5, 0.5)
\psline[showpoints=true](2.5,4)(4.5,0.5)
\psline[showpoints=true](0.5,1.5)(2.5,4)
\psline[showpoints=true](2.5,1)(0.5,3.5)
\psline[showpoints=true](2.5,1)(4.5,4.5)
\end{pspicture}
}
&
\resizebox{0.2\textwidth}{!}{
\begin{pspicture}(0,0)(4,4)
\psline[showpoints=true](1,1)(1,2)(1,3)
\psline[showpoints=true](2,1)(2,2)(2,3)
\psline[showpoints=true](3,1)(3,2)(3,3)
\psline[showpoints=true](1,1)(2,1)(3,1)
\psline[showpoints=true](1,2)(2,2)(3,2)
\psline[showpoints=true](1,3)(2,3)(3,3)
\pscurve(1,1)(2,2)(3,3)
\pscurve(1,2)(2,3)(3.5,3.5)(3,1)
\pscurve(3,2)(2,1)(0.5,0.5)(1,3)
\pscurve(3,1)(2,2)(1,3)
\pscurve(2,1)(1,2)(0.5,3.5)(3,3)
\pscurve(2,3)(3,2)(3.5,0.5)(1,1)
\end{pspicture}
}
&
\resizebox{0.2\textwidth}{!}{
\begin{pspicture}(0,0)(4,4)
\psline[showpoints=true](0.5,0.5)(0,2)(0.5,3.5)(2,4)(3.5,3.5)(4,2)(3.5,0.5)(2,0)(0.5,0.5)
\psline(0.5,0.5)(0.5,3.5)
\psline(0.5,0.5)(2,4)
\psline(0.5,0.5)(4,2)
\psline(0.5,0.5)(3.5,0.5)
\psline(0.5,0.5)(2,0)
\psline(0,2)(2,4)
\psline(0,2)(3.5,3.5)
\psline(0,2)(3.5,0.5)
\psline(0,2)(2,0)
\psline(0.5,3.5)(3.5,3.5)
\psline(0.5,3.5)(4,2)
\psline(0.5,3.5)(2,0)
\psline(2,4)(4,2)
\psline(2,4)(3.5,0.5)
\psline(3.5,3.5)(3.5,0.5)
\psline(3.5,3.5)(2,0)
\psline(4,2)(2,0)
\end{pspicture}
}
\end{tabular}
\end{center}
\caption{Examples of balanced and unbalanced configurations. From left to right: 
Fano plane ($v=b=d=7$ $r=k=3$),  
Pappus' configuration ($v=b=d=7$ $r=k=3$), 
Affine plane of order $3$ ($v=9$ $b=12$ $d=3$ $r=4$ $k=3$), 
$6$\hyp{}regular graph on 8 vertices ($v=8$ $b=24$ $d=8$ $r=6$ $k=2$)}
\label{fig:conf}
\end{figure}
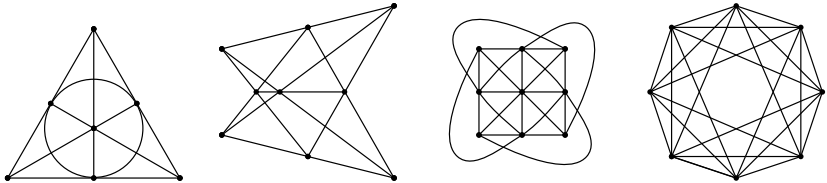
The following necessary conditions for the existence of configurations are well-known. 
\begin{lemma}
The lower bound of the number of points $v$ of an $(r,k)$\hyp{}configuration is $v\geq r(k-1)+1$, and the lower bound of the number of lines $b$ is $b\geq k(r-1)+1$. 
Also, the parameters $v,b,r,k$ always satisfy $vr=bk$. 
\end{lemma}
We say that a parameter set satisfying these two conditions are admissible. 
In general it is difficult to, given some admissible parameter set, determine if there exists some combinatorial configuration with these parameters. 
If this is the case, then we say that the parameter set is configurable.  
The (point) deficiency of a configuration with parameters $(v_r,b_k)$ is the difference $\delta_p=v-[r(k-1)+1]$, and the line deficiency is the difference $\delta_l=b-[k(r-1)+1]$. 
In balanced configurations the two deficiencies are equal. 

In 1886 Martinetti studied the construction of $v_3$\hyp{}configurations through the addition of a point and a line to existing $v_3$\hyp{}configurations~\cite{martinetti_1887,Grunbaum}. 
The construction is as follows.
Start with a $v_3$\hyp{}configuration and
assume that there are two parallel lines $\{a,b,c\}$ and $\{a',b',c'\}$ such that  $a$ and $a'$ are not collinear.
Add a point $p$ and replace the two parallel lines with the lines $\{p,b,c\}$, $\{p,b',c'\}$, $\{p,a,a'\}$. 
The result is a $(v+1)_3$\hyp{}configuration.
This construction is illustrated in Figure~\ref{fig:3}. 
We call such a construction a (Martinetti) augmentation. 
The inverse construction gives the smaller configuration from the larger one through the removal of one  point and one line. 
We call the inverse construction a (Martinetti) reduction. 
Martinetti called a configuration irreducible if it could not be constructed from another configuration through an augmentation. In other words, a configuration is irreducible if it does not allow a reduction.


In Martinetti's original paper he gave two infinite families of irreducible $v_3$\hyp{}configurations. One consisted of the cyclic configurations with base line $\{0,1,3\}$, starting with the smallest $v_3$\hyp{}configuration, the Fano plane. There is therefore at least one irreducible $v_3$ configuration for each $v\geq 7$. The other family gives one irreducible $(10n)_3$\hyp{}configuration for each $n\in\mathbb{Z}$, starting with Desargues' configuration. Martinetti claimed that these families of configurations were the only irreducible $v_3$\hyp{}configurations, with the addition of three sporadic examples for $v\leq 10$; more precisely,  Pappus' $(9_3)$\hyp{}configuration and two other $10_3$\hyp{}configurations.   
In 2007, Boben published a correction of this list, in which the two sporadic irreducible $10_3$\hyp{}configurations were shown to be the first elements in two additional infinite families of irreducible $(10n)_3$\hyp{}configurations, showing that there are four infinite families of irreducible $v_3$\hyp{}configurations~\cite{boben_graphs}. 

\begin{theorem}[Martinetti - Boben]
The list of (Martinetti) irreducible configurations are 
\begin{itemize}
\item the cyclic configurations with base line $\{0,1,3\}$. The smallest configuration in this family is the Fano plane,
\item the three infinite families $T_1(n)$, $T_2(n)$, $T_3(n)$, on $10n$ points. The smallest configuration in $T_1(n)$ is Desargues' configuration, and 
\item Pappus' configuration. 
\end{itemize}
\end{theorem} 

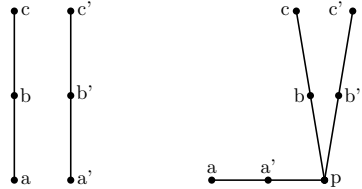
\begin{figure}
\begin{center}

\resizebox{0.2\textwidth}{!}{
\begin{pspicture}(0,0)(3,6)
\psline[showpoints=true](0,0)(0,1.5)(0,3)
\rput(0.2,0){a}
\rput(0.2,1.5){b}
\rput(0.2,3){c}
\psline[showpoints=true](1,0)(1,1.5)(1,3)
\rput(1.25,0.05){a'}
\rput(1.25,1.55){b'}
\rput(1.25,3.05){c'}
\psline[showpoints=true](5.5,0)(5.25,1.5)(5,3)
\rput(5.7,0){p}
\rput(5.05,1.5){b}
\rput(4.8,3){c}
\psline[showpoints=true](5.5,0)(4.5,0)(3.5,0)
\rput(4.5,0.25){a'}
\rput(3.5,0.2){a}
\psline[showpoints=true](5.5,0)(5.75,1.5)(6,3)
\rput(6,1.5){b'}
\rput(5.7,3.05){c'}
\end{pspicture}
}
\end{center}
\caption{Martinetti's augmentation construction. To the left the two parallel lines in the original $v_3$\hyp{}configuration, to the right the new incidences in the constructed $(v+1)_3$\hyp{}configuration. }
\label{fig:3}
\end{figure}

It results that, of several possible constructions of $(v+1)_3$ configurations from $v_3$\hyp{}configurations, Martinetti's construction is just one example. In 2000, Carstens et al. presented a rather complex set of reductions for which they claimed that the only irreducible configuration was the smallest $v_3$ configuration - the Fano plane~\cite{carstens}. 
However, in 2003, Ravnik used computer calculations to show that (at least) the Desargues configuration is also irreducible with respect to this set of reductions~\cite{erratum}. 
In 2005, Boben presented a simpler definition of reduction in terms of the Levi graph of the configuration. 
The Levi graph is a lossless representation of the incidences of the points and lines in form of a bipartite graph of girth at least six, and if $r=k=3$, then it is a cubic graph. 
In  \cite{boben_arxiv}, a reduction by the point $p$ and the line $l$ of the $v_3$\hyp{}configuration with Levi graph $G_v$ is defined as the Levi graph $G_{v-1}$ of a $(v-1)_3$\hyp{}configuration obtained from $G_v$ by removing the point vertex $p$ and the line vertex $l$ from $G_v$ and then connecting their neighbors in such way that the result remains cubic and bipartite. We call this construction a (Boben) reduction. A configuration is (Boben) irreducible if it does not admit a (Boben) reduction. 
%
%
According to Boben, with respect to this reduction, there are only two irreducible $v_3$\hyp{}configurations. 
\begin{theorem}[Boben]
The only (Boben) irreducible $v_3$\hyp{}configurations are the Fano plane and the Pappus configuration. 
\end{theorem}

This article presents a generalization of Boben's reduction to $(r,k)$\hyp{}configurations for any  $r,k\geq 2$, elaborates on the augmentation of $v_3$ and $v_4$\hyp{}configurations and provides some results that ensure irreducibility or reducibility in the general case.  
Augmentation and reductions of configurations are particularly interesting for the purpose of counting  configurations. 

\section{Reducibility of balanced configurations}
Balanced configurations are better studied than unbalanced configurations. This section presents results on augmentation and reduction constructions for balanced configurations. 

\subsection{Augmentation of balanced configurations}
The construction presented next is an augmenting construction for balanced $v_k$\hyp{}configurations. 
\begin{definition}
\label{def:augm_bal}
Let $C_v=(P,L,I)$ be a $v_k$\hyp{}configuration. 
Assume that there is a subset of $k$ points $Q\subseteq P$ and a subset of $k$ lines $M\subseteq L$,  such that 
\begin{itemize}
\item there is a bijection $f:Q\rightarrow M$ such that the image of a point $q$ is a line $f(q)$ through that point,
\item two points $q,q'\in Q$ either are not collinear, or are collinear only on the line $f(q)$ or on the line $f(q')$, and´
\item two lines $m,m'\in M$  either do not meet, or meet only in the point $f^{-1}(m)$ or in the point $f^{-1}(m')$. 
\end{itemize}  
Then there is a $(v+1)_k$\hyp{}configuration $C_{v+1}$, constructed from $C_v$ through the following augmentation procedure:

\noindent For all $q$ in $Q$, replace each incidence $(q,f(q))$ with 
\begin{itemize}
\item the incidence $(p,f(q))$, where $p$ is a new point, and 
\item the incidence $(q,l)$, where $l$ is a new line. 
\end{itemize}
\end{definition}
\begin{proposition}
\label{prop:augm_bal}
The result of the above construction is a $(v+1)_k$\hyp{}configuration $C_{v+1}$. 
\end{proposition}
\begin{proof}
In $C_v$, two points $q,q'\in Q$ are either not collinear, or collinear on $f(q)$ or $f(q')$. Since the incidences $(q,f(q))$ and $(q',f(q'))$ have been removed in $C_{v+1}$, it is clear that in $C_{v+1}$, the points in $Q$ are collinear only once, on the line $l$. 
Analogously, in $C_v$ two lines $m,m'\in M$ either do not meet, or meet only in the point $f^{-1}(m)$ or in the point $f^{-1}(m')$. Since the incidences $(m,f^{-1}(m))$ and $(m',f^{-1}(m'))$ have been removed in $C_{v+1}$, it is clear that in $C_{v+1}$ the lines in $M$ meet only once, in $p$. 
This also shows that any point in $C_{v+1}$ is collinear with $p$ at most once, and that any line in $C_{v+1}$ meets $l$ at most once. Indeed, the points in $C_{v+1}$ that are collinear with $p$ are the points on the lines in $M$, and since these lines only meet once in $C_{v+1}$, we see that any point in $C_{v+1}$ is collinear with $p$ at most once. Also, the lines in $C_{v+1}$ that meet $l$ are the lines through the points in $Q$, and since these points are collinear only once, in $l$, we see that any line in $C_{v+1}$ meets $l$ at most once.   
Now, these are the only incidences affected by the construction, and consequently, it is proved that $C_{v+1}$ is a partial linear space with $v+1$ points and $v+1$ lines. 
Finally, it is clear that there are $k$ points on each line and $k$ lines through every point, so that $C_{v+1}$ is a $(v+1)_k$\hyp{}configuration. 
\end{proof}
\begin{remark}
\label{remark:1}
The observant reader will find that there are other augmentation constructions which cannot be directly realized by following the steps described above. However, if we allow a final swapping of the incidences involved in the construction, then also these constructions may be described using Proposition~\ref{prop:augm_bal}.  
One example of this is Martinetti's augmentation. 
Consider $Q=\{a,a',a''\}$ and $M=\{\{a,b,c\},\{a',b',c'\},\{a'',b'',c''\}\}$, such that $\{a,b,c\}$ and $\{a',b',c'\}$ are parallel lines and  $a$ and $a'$ are not collinear, and no restrictions other than those in Proposition~\ref{prop:augm_bal} are put on $a''$ and $\{a'',b'',c''\}$,  and define $f(a)=\{a,b,c\}$, $f(a')=\{a',b',c'\}$ and $f(a'')=\{a'',b'',c''\}$. 
Replace the ocurrences of the points in $Q$ on the lines in $M$ with incidences to a new point $p$ so that the resulting lines are $\{p,b,c\}$, $\{p,b',c'\}$, $\{p,b'',c''\}$, and put the points in $Q$ on a new line $\{a,a',a''\}$. Now swap the incidences $(p,\{p,b'',c''\}$ and $(a'',\{a,a',a''\})$ to obtain Martinetti's construction. We see that the original line $\{a'',b'',c''\}$ is then left untouched, in consistency with the fact that Martinetti's construction only involved two lines.  
\end{remark}

Using Proposition~\ref{prop:augm_bal} it is not difficult to prove the following well-known result. 
\begin{corollary}
There is a $v_3$\hyp{}configuration for all admissible parameters. 
\end{corollary}
\begin{proof}
Any $v_3$\hyp{}configuration is augmentable.
Indeed,  if the $v_3$\hyp{}configuration has a triangle, then its three points $Q=\{q_1,q_2,q_3\}$ and its three lines $M=\{m_1,m_2,m_3\}$ together with the map $f(q_i)=m_i$ satisfy the conditions in Proposition~\ref{prop:augm_bal}. For an illustration of the augmentation in this case, see Figure~\ref{fig:augm_triangle}. If the configuration has no triangles, then consider a
path starting at a point $q_1$ of three lines $l_1$, $l_2$ and $l_3$, intersecting in two points $q_2$ and $q_3$. Then
$Q= {q_1, q_2, q_3}$ and $M=\{m_1,m_2,m_3\}$ satisfy the conditions of Proposition~\ref{prop:augm_bal}. 
Therefore there is a $(v+1)_3$\hyp{}configuration whenever there is a $v_3$\hyp{}configuration. 
The smallest $v_3$\hyp{}configuration is the Fano plane, with $v = 7$, and the result follows. 
\end{proof}
When $k$ is larger than 3, the situation is more complex. Indeed, the projective plane of order 3 is a $13_4$\hyp{}configuration which is not augmentable. However, if a $v_4$\hyp{}configuration has at least deficiency one, then it is augmentable. 
\begin{corollary}
There is a $v_4$\hyp{}configuration for all admissible parameters. 
\end{corollary}
\begin{proof}
Any $v_4$\hyp{}configuration of deficiency at least one is augmentable. 
Indeed, if the deficiency is at least one, then there are points $Q=\{q_1,q_2,q_3,q_4\}$ and lines $M=\{m_1,m_2,m_3,m_4\}$ forming either a quadrangle with $M$ as sides and $Q$ as vertices, or an open path $q_1m_1q_2m_2q_3m_3q_4m_4$ such that the conditions of Proposition~\ref{prop:augm_bal} are satisfied.
Therefore there is a $(v+1)_4$\hyp{}configuration whenever there is a $v_4$\hyp{}configuration. 
The smallest $v_4$\hyp{}configuration is the projective plane of order 3, and there exists also a $14_4$\hyp{}configuration. This latter configuration has deficiency one, and the result follows.  
\end{proof}

\begin{figure}
\begin{center}
\resizebox{0.5\textwidth}{!}{
\begin{pspicture}(8.5,4)
\pscircle[linestyle=dotted](7,1.666){1.21}
\psline[showpoints=true](0,1)(1,1)(3,1)
\psline[showpoints=true](2,3)(3,1)(3.5,0)
\psline[showpoints=true](1,1)(2,3)(2.5,4)
\pscurve[showpoints=true](5,1)(6,1)(7,1.666)
\pscurve[showpoints=true](7,1.666)(8,1)(8.5,0)
\pscurve[showpoints=true](7,1.666)(7,2.867)(7.5,4)
\psdots[fillcolor=black,dotstyle=square,
    dotsize=8pt](7,1.6666)
\end{pspicture}
}
\end{center}
\caption{The new (squared) point and the new (plotted) line of a $(v+1)_3$\hyp{}configuration to the right, added to a triangle in the original $v_3$\hyp{}configuration to the left.}
\label{fig:augm_triangle}
\end{figure}
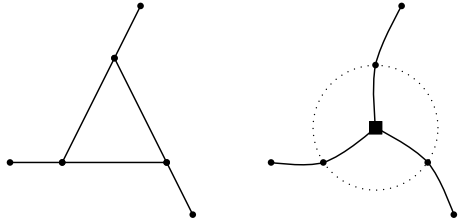

\subsection{Reduction of balanced configurations}
The inverse of the augmentation construction is the reduction. 

\begin{definition}
\label{def:irr_bal}
A reduction of a balanced configuration $(P,L,I)$ is a triple $(p,l,f')$ where
$p$ is a point, $l$ is a line, and $f'$ is an injective function $f' : Q' \rightarrow M'$,
where
\begin{itemize}
\item $Q' = \{q\in P : q \in l \mbox{ and } q \neq p\}$, and
\item $M' = \{m \in L: p \in m \mbox{ and } m \neq l\}$,
\end{itemize}
such that $q$ is not collinear with $r \in f'(q)$, except possibly through $l$ or with $p$. 
A configuration is reducible if it admits a reduction. 
Otherwise it is irreducible.
\end{definition}

\begin{lemma}
If a configuration $(P,L,I)$ admits a reduction as in Definition~\ref{def:irr_bal}, then there is a reduced configuration $(P\setminus\{p\},L\setminus\{l\},\tilde{I})$ obtained from $(P,L,I)$ by replacing the incidences $(p,f'(q))$ and $(q,l)$ for $q\in Q'$ with the incidences $(q,f'(q))$ and removing the point $p$ and the line $l$.
\end{lemma}
\begin{proof}
Each point is on the same number of lines, and each line goes through the same number of points in $(P\setminus\{p\},L\setminus\{l\},\tilde{I})$ as in $(P,L,I)$. The definition of $f'$ ensures that any two lines in $(P\setminus\{p\},L\setminus\{l\},\tilde{I})$ meet in at most one point. 
\end{proof}

\begin{lemma}
The reduction is the inverse construction of the augmentation. 
\end{lemma}
\begin{proof}
Let $C_v=(P,L,I)$ be a $v_k$\hyp{}configuration with a set $Q=\{q_1,\dots,q_k\}$ of $k$ points and a set $M=\{m_1,\dots,m_k\}$ of $k$ lines satisfying the requirements in Proposition~\ref{prop:augm_bal}. 
Consider the incidences in the augmented $(v+1)_k$\hyp{}configuration $C_{v+1}$ which are not in $C_v$. These incidences are $(p,f(q_i))$ and $(q_i,l)$, for $i\in \{1\dots k\}$. Also consider the incidences that were removed from $C_v$ in the construction of $C_{v+1}$, $(q_i,f(q_i))$,  for $i\in \{1\dots k\}$.  
As described in Remark~\ref{remark:1} and Remark~\ref{remark:2},  some of the incidences involved in the augmentation may be swapped afterwards. This is only relevant if the incidences is of the form  $(p,f(q_i))$ and $(q_i,l)$ (which produces the incidence $(p,l)$, so that the new point and the new line are incident). 
In this case, let $Q'=Q\setminus \{q_i\}$ and $M'=M\setminus \{f(q_i)\}$, otherwise, let $Q'=Q$ and $M=M'$.   
Define the reduction $(p,l,f')$ with $f':Q'\rightarrow M'$ the restriction of $f$ to $Q'$. 
This is a well-defined reduction, since $q\in Q'$ is not collinear with any point $r$ on $f(q)$ in $C_{v+1}$ except possibly with $p$ or through $l$.  
Replace the incidences $(p,f(q))$ and $(q,l)$ for $q\in Q'$ with the incidences $(q,f(q))$ and remove the point $p$ and the line $l$. This reduction produces a $v_k$\hyp{}configuration with the same incidences as $C_v$, hence equal to $C_v$. 
\end{proof}

Observe that according to Definition~\ref{def:irr_bal}, a balanced configuration is irreducible exactly if it is impossible to remove one point and one line and obtain a new configuration, through modifications that only affect the incidences of the removed point and line. 
This is the same definition of irreducibility as the one used by Boben in the case of $v_3$\hyp{}configurations, although he expressed it in terms of the Levi graph. 
Martinetti's irreducibility is the special case in which the removed point $p$ is on the removed line $l=\{p,a,a'\}$, so that $Q'=\{a,a'\}$, $M'=\{\{p,b,c\},\{p,b',c'\}\}$ and $f':Q'\rightarrow M'$ is defined by $f'(a)=\{p,b,c\}$ and $f'(a')=\{p,b',c'\}$. 
The reduction then consists in removing $p$ and $l$ and replacing the appearances of $p$ in $m\in M$ with $f'^{-1}(m)$. 
Note that no incidence swapping was needed when describing Martinetti's reduction in terms of Definition~\ref{def:irr_bal}.  

The somewhat awkward definition of reducibility for balanced configurations can  also be restated as follows.  
\begin{corollary}
\label{cor:irr}
A balanced configuration $v_k$ is reducible if and only if it contains one line $l$ and one point $p$, such that the points $q_i$ on $l$ and the lines $m_i$ through $p$ can be labelled so that $q_i$ is not collinear with any point on $m_i$ except possibly through $l$ or with $p$, for $i\in [1,k]$.  
\end{corollary}
\begin{proof}
Indeed, the function $f(p_i)=l_i$ for $p_i\neq p$ gives a reduction $(p,l,f)$. 
\end{proof}

\begin{remark}
\label{remark:2}
The general form of the augmentation and reduction constructions implies that the resulting configuration may fail to be connected. However, there is choice in the constructions. It is always possible to make the resulting configuration connected. In practice, this can be achieved by swapping two incidences located in different connected components, as described for example in \cite{Riemann}. That is, if $(p,q)$ and $(p',q')$ are two incidences in two different connected components, then replace these incidences with $(p,q')$ and $(p',q)$. By repeating this process as long as the configuration have at least two connected components, eventually a connected configuration is obtained. If the incidences $(p,q)$ and $(p',q')$ are not incidences of the old configuration, but instead both come from the augmentation or the reduction construction, then the incidence swapping gives a configuration that would have resulted from another choice in the construction. Note that Martinetti's augmentation is described in this way in Remark~\ref{remark:1}. 
\end{remark}

\section{Unbalanced configurations}

It is not possible to reduce unbalanced configurations through the removal of one point and one line. 
This is a consequence of the necessary condition for the existence of a configuration $vr=bk$. 
Indeed, $vr=bk$ implies that $(v-1)r/k=vr/k - r/k = b - r/k$ so that $(v-1)r\neq (b-1)k$, whenever $r\neq k$.  
In this context, the reduced parameter set $(d,r,k)$ is useful - the parameter set $(d,r,k)$ is admissible for every integer $d$ satisfying $d\geq \gcd(r,k)(r(k-1)+1)/k$.  
Therefore, a reduction should, given a $(d,r,k)$\hyp{}configuration, produce a $(d-1,r,k)$\hyp{}configuration through the removal of an appropiate number of points and lines, using only modifications that affect the incidences of these removed points and lines.
More precisely, the number of points to remove is $k/\gcd(r,k)$ and the number of lines is $r/\gcd(r,k)$. 

\subsection{Augmentation of unbalanced configurations}
In \cite{Riemann} we described a construction of a  $(d_1+\cdots+d_n+1,r,k)$\hyp{}configuration from $n$ configurations with parameters  $(d_1,r,k),\dots,(d_n,r,k)$. By applying this construction to a single configuration with parameters $(d,r,k)$, one obtains a $(d+1,r,k)$\hyp{}configuration through an augmentation construction. The requirement for this construction to work is that the original configuration contains  a set of $rk/\gcd(r,k)$ points $Q$ and a set of $rk/\gcd(r,k)$ lines $M$ with a special property. 

\begin{definition}
\label{def:augm_unbal}
Let $C_d=(P,L,I)$ be a $(d,r,k)$\hyp{}configuration. 
Assume that there is a multiset $Q$ of $rk/\gcd(r,k)$ (not necessarily distinct) points in $P$ and a multiset $M$ of $rk/\gcd(r,k)$ (not necessarily distinct) lines in $L$ such that 
\begin{itemize}
\item there is a bijection $f:Q\rightarrow M$ such that the image of a point $q$ is a line $f(q)$ through that point,
\item $Q$ can be partitioned into $r/\gcd(r,k)$ parts, each of cardinality $k$, such that two points $q$ and $q'$ in each part, either are not collinear, or are collinear only on the line $f(q)$ or on the line $f(q')$, and
\item $M$ can be partitioned into $k/\gcd(r,k)$ parts, each of cardinality $r$,  such that two lines $m$ and $m'$ in each part either do not meet, or meet only in the point $f^{-1}(m)$ or in the point $f^{-1}(m')$. 
\end{itemize}  
Then there is a $(d+1,r,k)$\hyp{}configuration, constructed from $C_d$ through the following augmentation procedure:

\noindent For all $q$ in $Q$, replace each incidence $(q,f(q))$ with 
\begin{itemize}
\item the incidence $(p,f(q))$, where $p$ is a point from a set $R$ of $k/\gcd(r,k)$ new points, in a way that ensures that each point in $N$ is on exactly $r$ lines, and 
\item the incidence $(q,l)$, where $l$ is a line from a set $N$ of $r/\gcd(r,k)$ new lines, in a way that ensures that each line in $N$ contains exactly $k$ points. 
\end{itemize}
\end{definition}

\begin{proposition}
\label{prop:augm_unbal}
The result of the above construction is a $(d+1,r,k)$\hyp{}configuration. 
\end{proposition}

The proof of Proposition~\ref{prop:augm_unbal} is only slightly more involved than the proof of Proposition~\ref{prop:augm_bal}, which is the special case $r=k$. For more details of in the general case, see \cite{Riemann}.
\begin{example}
The finite affine plane of order 3 is a $(3,4,3)$\hyp{}configuration $(P,L,I)$  with 9 points and 12 lines (see Figure~\ref{fig:conf}). Label the points $P$ as $1,\dots,9$ so that the lines $L$ are 
$$\{1,2,3\},\{4,5,6\},\{7,8,9\},\{1,4,7\},\{2,5,8\},\{3,6,9\},$$
$$\{2,4,9\},\{3,5,7\},\{1,6,8\},\{3,4,8\},\{1,5,9\},\{2,6,7\}.$$ 
An augmentation requires 12 points and 12 lines, and we use $M=L$, $Q$ the multiset consisting of $P$ with the three points $1,2,9$ repeated, and the bijection $f:Q\rightarrow M$ defined by 
$$\begin{array}{llllll}
f(1_1)&=\{1,2,3\}&
f(1_2)&=\{1,6,8\}&
f(2_1)&=\{2,4,9\}\\
f(2_2)&=\{2,6,7\}&
f(3)&=\{3,5,7\}&
f(4)&=\{3,4,8\}\\
f(5)&=\{1,5,9\}&
f(6)&=\{4,5,6\}&
f(7)&=\{1,4,7\}\\
f(8)&=\{2,5,8\}&
f(9_1)&=\{3,6,9\}&
f(9_2)&=\{7,8,9\}
\end{array}$$
where $x_1$ and $x_2$ denotes the first and the second occurrence of $x$ in $Q$. 
This gives, with the new points $p_1,p_2,p_3$ and the new lines $l_1,l_2,l_3,l_4$,  a $(4,4,3)$\hyp{}configuration with 12 points $1,\dots,9,p_1,p_2,p_3$ and 16 lines 
$$\begin{array}{llll}
\{p_1,1,4\}&
\{p_1,2,3\}&
\{p_1,5,7\}&
\{p_1,6,8\}\\

\{p_2,2,5\}&
\{p_2,6,7\}&
\{p_2,3,8\}&
\{p_2,4,9\}\\

\{p_3,3,6\}&
\{p_3,4,5\}&
\{p_3,1,9\}&
\{p_3,7,8\}\\

\{1,3,7\}=l_1&
\{2,4,8\}=l_2&
\{5,6,9\}=l_3&
\{1,2,9\}=l_4.
\end{array}
$$
The partition of $Q$  was
$$\{1,3,7\},
\{2,4,8\},
\{5,6,9\},
\{1,2,9\}$$ 
and the partition of $M$ was
$$\{\{1,2,3\},\{1,4,7\},\{3,5,7\},\{1,6,8\}\},$$
$$\{\{2,4,9\},\{2,5,8\},\{3,4,8\},\{2,6,7\}\},$$
$$\{\{3,6,9\},\{4,5,6\},\{1,5,9\},\{7,8,9\}\}.$$
\end{example}

\subsection{Reduction of unbalanced configurations}
The inverse of the augmentation construction is a reduction, which is a generalization of the reduction described in Definition~\ref{def:irr_bal}. 

\begin{definition}
\label{def:irr_unbal}
A reduction of an unbalanced configuration $(P,L,I)$ is a triple $(R,N,f')$ where
$R$ is a set of $k/\gcd(r,k)$  points, $N$ is a set of $r/\gcd(r,k)$ lines, and $f'$ is a bijection between multisets $f' : Q' \rightarrow M'$,
where
\begin{itemize}
\item $Q' = \{q\in P : \exists l\in N: q \in l \mbox{ and } q \not\in R\}$, and
\item $M' = \{m \in L: \exists p\in R : p\in m \mbox{ and } m \not\in N\}$,
\end{itemize}
such that $q$ is not collinear with $r \in f(q)$, except possibly through one of the  lines in  $N$ or with one of the points in $R$. Both $Q'$ and $M'$ are multisets and as such they may contain some element more than once. 
A configuration is reducible if it admits a reduction. 
Otherwise it is irreducible.
\end{definition}

\begin{lemma}
If a configuration $(P,L,I)$ admits a reduction as in Definition~\ref{def:irr_unbal}, then there is a reduced configuration $(P\setminus R,L\setminus N,\tilde{I})$ obtained from $(P,L,I)$ by replacing the incidences $(p,f'(q))$ and $(q,l)$ for $q\in Q'$ with the incidences $(q,f'(q))$ and removing the points in $R$ and the lines in $N$.
\end{lemma}
\begin{proof}
Each point is on the same number of lines, and each line goes through the same number of points in $(P\setminus\{p\},L\setminus\{l\},\tilde{I})$ as in $(P,L,I)$. The definition of $f'$ ensures that any two lines in $(P\setminus R,L\setminus N,\tilde{I})$ meet in at most one point. 
\end{proof}

\begin{lemma}
The reduction is the inverse construction of the augmentation. 
\end{lemma}
\begin{proof}
Let $C_d$ be a $(d,r,k)$\hyp{}configuration with a set $Q=\{q_1,\dots,q_{rk/\gcd(r,k)}\}$ of points, a set $M=\{m_1,\dots,m_{rk/\gcd(r,k)}\}$ of lines and a bijection $f:Q\rightarrow M$, satisfying the requirements of Definition~\ref{def:augm_unbal}. 
Consider the incidences in the augmented $(d+1,r,k)$\hyp{}configuration $C_{d+1}$ which are not in $C_d$. 
These incidences are $(p,f(q_i))$ and $(q_i,l)$, for $i\in \{1\dots rk/\gcd(r,k)\}$, for some $p\in R$ and some $l\in N$ (if no swapping is allowed).  
Also consider the incidences that were removed from $C_d$ in the construction of $C_{d+1}$, $(q_i,f(q_i))$,  for $i\in \{1\dots rk/\gcd(r,k)\}$.  
If we allow, for some set of indices $I$, the incidences $(q_i,l)$ and $(p,f(q_i))$, $i\in I$, to be swapped afterwards, making the lines in $N$ and the points in $R$ incident,  
then let $Q'=Q\setminus \{q_i: i\in I\}$ and $M'=M\setminus \{f(q_i)\in I\}$, otherwise, let $Q'=Q$ and $M=M'$.   
Define the reduction $(R,N,f')$ with $f':Q'\rightarrow M'$ the restriction of $f$ to $Q'$. 
This is a well-defined reduction, since $q\in Q'$ is not collinear with any point $r$ on $f(q)$ in $C_{d+1}$ except possibly with some $p\in R$ or through some $l\in N$.   
For all $p\in R$ and all $l\in N$, replace the incidences $(p,f(q))$ and $(q,l)$ for $q\in Q'$ with the incidences $(q,f(q))$ and remove the point $p$ and the line $l$. 
This reduction produces a $(d,r,k)$\hyp{}configuration with the same incidences as $C_d$, hence equal to $C_d$. 
\end{proof}

\begin{remark}
Remark~\ref{remark:2}, regarding the connectedness of the result of the augmentation and the reduction constructions, is valid also for unbalanced configurations.
\end{remark}

\section{Irreducibility and reducibility in configurations}
We would like to characterize the set of irreducible configurations. The results presented next provide some progress in this direction. 

\subsection{Irreducibility in small configurations}
The smallest $(r,k)$\hyp{}configurations  are the linear spaces, whenever they exist. Examples of linear spaces are projective and affine planes.  
The inexistence of smaller $(r,k)$\hyp{}configurations clearly implies that the linear spaces are irreducible. 
However, as the next results states, there are also  other $(r,k)$\hyp{}configurations that are necessarily irreducible because they are small. 
\begin{lemma}
\label{thm:1}
Any $(r,k)$\hyp{}configuration with point deficiency $\delta_p<k-(r+k)/\gcd(r,k)$ or line deficiency $\delta_l<r-(r+k)/\gcd(r,k)$ is irreducible. 
\end{lemma}
\begin{proof}
In a reducible configuration there are points $Q'$ and lines $M'$ and a bijection $f':Q' \rightarrow M'$ such that $q\in Q'$ is not collinear with any of the $k$ points on $f'(q)\in M'$, except possibly with some of the $k/\gcd(r,k)$ removed points $R$, or through some of the $r/\gcd(r,k)$ removed lines $N$. 
This condition is equivalent to requiring that $f'(q)\in M'$ does not meet any of the $r$ lines through $q$, 
except possibly on some of the $k/\gcd(r,k)$ removed points $R$, 
or through some of the $r/\gcd(r,k)$ removed lines $N$. 
But, if the point deficiency $v-[r(k-1)+1]$ is smaller than $k-(r+k)/\gcd(r,k)$, 
then for any point $q$ there is no line $m$ such that $q$ is only collinear with the points on $m$ on either some points in $R$ or through some lines in $N$, so the configuration must be irreducible. 
Analogously, if the line deficiency $b-[k(r-1)+1]$ is smaller than $r-(r+k)/\gcd(r,k)$, then for any line $m$ there is no point $q=f'^{-1}(m)$ such that $m$ does not meet any of the points through $q$, except if it is a line in $N$ or if the intersection point is a point in $R$, and again, the configuration must be irreducible. 
\end{proof}
This bound is sharp in the meaning that there are reducible $(r,k)$\hyp{}configurations of deficiency $\delta_p=k-(r+k)/\gcd(r,k)$ and $\delta_l=r-(r+k)/\gcd(r,k)$. For example, the M\"obius-Kantor $8_3$\hyp{}configuration, with deficiency $\delta_p=\delta_l=3-(3+3)/3=1$, is reducible.   
Indeed, for $v_3$\hyp{}configurations, Lemma~\ref{thm:1} is only relevant for deficiency 0.    
From~\cite{boben_arxiv} we know that the irreducible $v_3$\hyp{}configurations are the Fano plane (of deficiency 0) and the Pappus' configuration. However, Pappus' configuration has deficiency 3, so Lemma~\ref{thm:1} does not apply. 
But when $k$ is larger than 3, then Lemma~\ref{thm:1} can imply the irreducibility of more than one $(r,k)$\hyp{}configuration. 
Indeed, for $r=k=4$, the two configurations $13_4$ and $14_4$ must be irreducible, for $r=k=5$, the configurations $21_5$ and $23_5$ have deficiency 0 and 2, so they are both irreducible. There is no configuration with parameters $22_5$. 
For $r=k=6$, the configurations $31_6$ and $34_6$ are both irreducible, since they have deficiencies 0 and 3, hence smaller than 4. There are no configurations $32_6$ 
and $33_6$. 
For $k=7$, the only configuration with deficiency strictly smaller than 5 that is known to exist is of deficiency 2, with parameters $45_7$. There are no configurations $43_7$ and $44_7$. If the configurations $46_7$ and $47_7$ exist, then they are irreducible. For a reference on the existence and non-existence of small balanced configurations, see for example~\cite{Gropp}.

\subsection{Irreducible configurations and transversality - Pappus' configuration}
The irreducibility of the Fano plane can be explained by Lemma~\ref{thm:1}. 
The reason why Pappus' configuration is irreducible is different, and based on transversality. 

A transversal design $TD_{\lambda}(k, n)$ is a $k$\hyp{}uniform incidence geometry on
$kn$ points, allowing a partition of $k$ groups of $n$ elements, such that
any group and any block contain exactly one common point, and
every pair of points from distinct groups is contained in exactly $\lambda$ blocks. 
A transversal design $TD_{\lambda}(k,n)$ is resolvable if the set of blocks can be partitioned into parallel
classes of blocks, such that each class forms a partition of the point set. 
 
When $\lambda=1$, then the design is a $(kn_n,n^2_k)$\hyp{}configuration, and we call the blocks lines.  
There is a $TD_1(k, n)$ whenever there is an affine plane of order $n$ and
$k \leq  n$.  
Indeed, just take the points on $k$ lines in a parallel class of the affine plane and restrict the rest of
the lines to these points. 
Pappus' configuration can be constructed in this way from the affine plane of order 3, by restricting to the points on all the 3 lines in one of its 4 classes of parallel lines. Since the points on these 3 lines are all points in the affine plane, in this case the construction consists of eliminating one parallel class of lines from the affine plane. 
By instead  restricting to the points on only two lines in one of the parallel classes, a transversal design $TD_1(2,3)$ is obtained, which is a $(6_3,9_2)$\hyp{}configuration, that is, the bipartite complete graph on 6 vertices. 
\begin{lemma}
A resolvable transversal design $TD_1(k, n)$ is irreducible if $$k\geq (k+r)/\gcd(r,k)+1.$$
\end{lemma}
\begin{proof}

Let $T=TD_1(k,n)$ be a resolvable transversal design.
Let $p$ be a point in $T$ and $m_1,\dots,m_r$ the lines through $p$.
Then $m_1,\dots , m_r$ are in different parallel classes.
Let $l$ be a line in $T$ and $q$ a point on $l$.
Then $q$ is collinear with all points on the lines $m_1,\dots ,m_r$ except one on
each line, which belong to the same group as $q$ ($q$ is not collinear with itself). 
At most $(r+k)/\gcd(r,k)$ of these incidences will not obstruct a reduction, since a reduction removes $k/\gcd(r,k)$ points and $r/\gcd(r,k)$ lines. 
Therefore,  since the point $p$ and the line $l$ were chosen arbitrarily, if $k\geq (k+r)/\gcd(r,k)+1$, then it is not possible to find a reduction of $T$ that removes $p$ and $l$ (and possibly other points and lines). 
More precisely, there is no $f'$ mapping $q$ to $m_i$, for some $i$, such that $q$ is not collinear with any point on $m_i$, except possibly with the $k/\gcd(r,k)$ removed points or through the $r/\gcd(r,k)$ removed lines.

\end{proof}
Note that this proves that Pappus' configuration, which is a $TD_1(3,3)$, is irreducible, but it does not prove the same fact for the $TD_1(2,3)$. Indeed, the latter is reducible, as is any graph with deficiency high enough.  
Observe that the deficiency of a transversal design $TD_1(k,n)$ satisfies $d=n-1\geq k-1$, so that these irreducible configurations are not covered by Lemma~\ref{thm:1}.

\subsection{Reducibility in large configurations}
When the deficiency is large enough, then reducibility can be ensured.
\begin{lemma}
A $(v_r,b_k)$\hyp{}configuration is reducible if
$b\geq 1+r+r(k-1)(r-1)+r(r-1)^2(k-1)^2$
\end{lemma}
\begin{proof}
Given a point $p$ there are at most $r+r(k-1)(r-1)+r(r-1)^2(k-1)^2$ lines containing at least one point at distance one or two from $p$. This bound is attained if the configuration is triangle-, quadrangle-, and pentagonal-free. 
If the configuration contains an additional line $l$, then $l$ contains only points at distance at least three from $p$. 
In other words, the points on $l$ are not collinear with any point that is collinear with $p$. 
This implies that the configuration is reducible. 
\end{proof}
In a balanced $v_k$\hyp{}configuration, the number of lines $b$ equals the number of points $v$. 
Therefore, in this case the bound also takes the form $v \geq  1 + k  +k (k - 1)^2 + k(k-1)^4$.
This is not a sharp bound, indeed, for $v_3$\hyp{}configuration it can only ensure reducibility for $v\geq 64$, but we know that all $v_3$\hyp{}configurations are reducible for $v\geq 10$. 

The irreducibility of $v_k$\hyp{}configurations with $v$ between these lower and upper bounds, is still in general an open question.  It is of course possible to test a given configuration, by hand or with the help of a computer.  However, for exact enumeration purposes it is of course interesting to have exact general results.  

\section{Conclusions}
We have seen that it is possible to define irreducibility not only for $(v_k)$ configuration, but for $(v_r,b_k)$\hyp{}configurations in general. 
Augmentation and reduction constructions for $(v_r,b_k)$\hyp{}configurations have been defined in a general manner, and several general results on augmentability and reducibility have been described.  
Irreducibility has been proved for configurations with point deficiency $\delta_p<k-(r+k)/\gcd(r,k)$ or line deficiency $\delta_l<r-(r+k)/\gcd(r,k)$, and for (some) transversal designs $TD_1(k,n)$. A $TD_1(k,n)$ has point deficiency $n-1=r-1$ and line deficiency $r^2-rk+k-1$. For $r=k=3$, these are the only irreducible configurations, and at this point, no other irreducible configurations are known in the general case. There is an upper bound for irreducibility requiring the number of lines to satisfy  $b< 1+r+r(k-1)(r-1)+r(r-1)^2(k-1)^2$. This bound is not sharp, and a better bound would probably set the point deficiency closer to $r$. 

The author is aware of at least two applications of augmentations and reductions of configurations. One is the enumeration of configurations, the other is the use of configurations in cryptography and coding theory. When a configuration is used to define a key-distribution scheme, and new parties join or leave, augmentation and reduction constructions can modify the structure while minimizing the costs of key-reassignment. 
However, it is important to be aware of the fact that the constructions described in this paper may fail to preserve interesting properties. 

\section*{Acknowledgements}
The author would like to thank Mark Ellingham for pointing out that not only the augmentation but also the reduction can result in a disconnected configuration. 
Partial financial support is acknowledged from the Spanish MEC project ICWT (TIN2012-32757) and ARES (CONSOLIDER INGENIO 2010 CSD2007-00004), as well as from the Swedish National Graduate School in Computer Science (CUGS).

\end{document}